\documentclass[a4paper,12pt]{amsart}

\usepackage{amssymb}
\usepackage{amsmath}
\usepackage{amssymb}
\usepackage[dvips]{color}
\usepackage{booktabs}
\usepackage{bbm}

\newtheorem{priteo}{Theorem}
\newtheorem{lema}{Lemma}

\newtheorem{defi}{Definition}

\begin{document}

\title{ The cyclic Hopf $H~\mathrm{mod}~K$ Theorem}

\author[ The cyclic Hopf $H~\mathrm{mod}~K$ Theorem \today]{Adrian C. Murza}

\address{Centro de Matem\'atica da Universidade do Porto.\\ Rua do Campo Alegre 687, 4169-007 Porto, Portugal}
\date{\today}
\thanks{This work was supported by FCT grant $SFRH/ BD/ 64374/ 2009$.}
\begin{abstract}
The $H~\mathrm{mod}~K$ theorem gives all possible periodic solutions in a $\Gamma-$equivariant dynamical system, based on the group-theoretical aspects. In addition, it classifies the spatio temporal symmetries that are possible. By the contrary, the equivariant Hopf theorem guarantees the existence of families of small-amplitude periodic solutions bifurcating from the origin for each $\mathbf{C}-$axial subgroup of $\Gamma\times\mathbb{S}^1.$ In this paper we identify which periodic solution types, whose existence is guaranteed by the $H~\mathrm{mod}~K$ theorem, are obtainable by Hopf bifurcation, when the group $\Gamma$ is finite cyclic.
\end{abstract}



\maketitle

\section{Introduction}\label{section Introduction}
In the formalism of equivariant differential equations \cite{GS85}, \cite{GS88} and \cite{GS03} have been described two methods for obtaining periodic solutions: the $H~\mathrm{mod}~K$ theorem and the equivariant Hopf theorem. While the $H~\mathrm{mod}~K$ theorem offers the complete set of possible periodic solutions based exclusively on the structure of the group $\Gamma$ acting on the differential equation, the equivariant Hopf theorem guarantees the existence of families of small-amplitude periodic solutions bifurcating from the origin for all $\mathbf{C}-$axial subgroups of $\Gamma\times\mathbb{S}^1.$

Not always all solutions predicted by the $H~\mathrm{mod}~K$ theorem can be obtained by the generic Hopf bifurcation \cite{GS03}. In \cite{Abelian Hopf} there are described which periodic solutions, whose existence is guaranteed by the $H~\mathrm{mod}~K$ theorem are obtainable by the Hopf bifurcation when the group $\Gamma$ is finite abelian. In this article, we pose a more specific question: what periodic solutions predicted by the $H~\mathrm{mod}~K$ theorem are obtainable by the Hopf bifurcation when the group $\Gamma$ is finite cyclic. We will answer this question by finding which additional constraints have to be added to the Abelian Hopf $H~\mathrm{mod}~K$ theorem \cite{Abelian Hopf} so that the periodic solutions predicted by the $H~\mathrm{mod}~K$ theorem coincide with the ones obtained by the equivariant Hopf theorem when the group $\Gamma$ is finite cyclic.

\section{The $H~\mathrm{mod}~K$ theorem}\label{section H mod K theorem}

We call $(\gamma,\theta)\in\Gamma\times\mathbf{S}^1$ a spatio-temporal symmetry of the solution $x(t).$ A spatio-temporal symmetry of $x(t)$ for which $\theta=0$ is called a spatial symmetry, since it fixes the point $x(t)$ at every moment of time. The group of all spatio-temporal symmetries of $x(t)$ is denoted
$$\Sigma_{x(t)}\subseteq\Gamma\times\mathbf{S}^1.$$
As shown in \cite{GS03}, the symmetry group $\Sigma_{x(t)}$ can be identified with a pair of subgroups $H$ and $K$ of $\Gamma$ and a homomorphism $\Theta:H\rightarrow\mathbf{S}^1$ with kernel $K.$ Define
\begin{equation}
    \begin{array}{l}
        K=\{\gamma\in\Gamma:\gamma x(t)=x(t)~~\forall t\}\\
        H=\{\gamma\in\Gamma:\gamma x(t)=\{x(t)\}~~\forall t\}.
    \end{array}
\end{equation}

The subgroup $K\subseteq\Sigma_{x(t)}$ is the group of spatial symmetries of $x(t)$ and the subgroup $H$ consists of those symmetries that preserve the trajectory of $x(t),$ ie. the spatial parts of the spatio-temporal symmetries of $x(t).$ The groups $H\subseteq\Gamma$ and $\Sigma_{x(t)}\subseteq\Gamma\times\mathbf{S}^1$ are isomorphic; the isomorphism is in fact just the restriction to $\Sigma_{x(t)}$ of the projection of $\Gamma\times\mathbf{S}^1$ onto $\Gamma.$ Therefore the group $\Sigma_{x(t)}$ can be written as
$$\Sigma^{\Theta}=\{  (h,\Theta(h))  :h\in H, \Theta(h)\in\mathbf{S}^1  \}.$$
Moreover, we call $\Sigma^{\Theta}$ a twisted subgroup of $\Gamma\times\mathbf{S}^1.$
In our case $\Gamma$ is a finite cyclic group and the $H~\mathrm{mod}~K$ theorem states necessary and sufficient conditions for the existence of a periodic solution to a $\Gamma-$ equivariant system of ODEs with specified spatio-temporal symmetries $K\subset H\subset \Gamma.$ Recall that the isotropy subgroup $\Sigma_x$ of a point $x\in\mathbb{R}^n$ consists of group elements that fix $x,$ that is they satisfy
$$\Sigma_x=    \sigma\in \Gamma :\sigma x=x.$$
Let $N(H)$ be the normalizer of $H$ in $\Gamma,$ satisfying $N(H)=\{\gamma\in\Gamma:\gamma H=H\gamma\}.$ Let also $\mathrm{Fix}(K)=\{x\in\mathbb{R}^n:kx=x~\forall k\in K\}.$

\begin{defi}\label{defi L}
Let $K\subset\Gamma$ be an isotropy subgroup. The variety $L_K$ is defined by
$$L_K=\bigcup_{\gamma\notin K}\mathrm{Fix}(\gamma)\cap\mathrm{Fix}(K).$$
\end{defi}

\begin{priteo}\label{teorema H mod K}
($H~\mathrm{mod}~K$ Theorem \cite{GS03}) Let $\Gamma$ be a finite group acting on $\mathbb{R}^n.$ There is a periodic solution to some $\Gamma-$equivariant system of ODEs on $\mathbb{R}^n$ with spatial symmetries $K$ and spatio-temporal symmetries $H$ if and only if the following conditions hold:
\begin{itemize}
\item [(a)] $H/K$ is cyclic;
\item [(b)] $K$ is an isotropy subgroup;
\item [(c)] $dim~Fix(K)\geqslant2.$ If $dim~Fix(K)=2,$ then either $H=K$ or $H=N(K)$;
\item [(d)] $H$ fixes a connected component of $\mathrm{Fix(K)\backslash L_K,}$ where $L_K$ appears as in Definition \ref{defi L} above;
\end{itemize}
Moreover, if $(a)-(d)$ hold, the system can be chosen so that the periodic solution is stable.
\end{priteo}
\begin{defi}
The pair of subgroups $(H,K)$ is called admissible if the pair satisfies hypotheses $(a)-(d)$ of Theorem \ref{teorema H mod K}, that is, if there exist periodic solutions to some $\Gamma-$equivariant system with $(H,K)$ symmetry.
\end{defi}

\section{Hopf bifurcation with cyclic symmetries}\label{Hopf bifurcation with cyclic symmetries}
In the following we recall two results from \cite{Abelian Hopf} needed later for the proof of the Theorem \ref{teorema cyclic Hopf}. Let $x_0\in\mathbb{R}^n.$ Suppose that $V$ is an $\Sigma_{x_0}-$invariant subspace of $\mathbb{R}^n.$ Let $\hat{V}=x_0+V,$ and observe that $\hat{V}$ is also $\Sigma_{x_0}-$invariant.
\begin{lema}\label{lema center subspace}
Let $g$ be an $\Sigma_{x_0}-$equivariant map on $\hat{V}$ such that $g(x_0)=0.$ Then $g$ extends to a $\Gamma-$equivariant mapping $f$ on $\mathbb{R}^n$ so that the center subspace of $(df)_{x_0}$ equals the center subspace of $(dg)_{x_0}.$
\end{lema}
\begin{proof}
See \cite{Abelian Hopf}.
\end{proof}
\begin{lema}\label{lema diffeomorphism}
Let $f:\mathbb{R}^n\rightarrow\mathbb{R}^n$ be $\Gamma-$equivariant and let $f(x_0)=0.$ Let $V$ be the center subspace of $(df)_{x_0}.$ Then there exists a $\Gamma-$equivariant diffeomorphism $\psi:\mathbb{R}^n\rightarrow\mathbb{R}^n$ such that $\psi(x_0)=x_0$ and the center manifold of the transformed vector field
$$\psi_*f(x)\equiv(d\psi)^{-1}_{\psi(x)}f(\psi(x))$$
is $\hat{V}.$
\end{lema}
\begin{proof}
See \cite{Abelian Hopf}.
\end{proof}
In order to state the Cyclic Hopf theorem, we need first the following lemma.
\begin{lema}
The group $\Gamma$ is cyclic if and only it is a homomorphic image of $\mathbb{Z}.$
\end{lema}
\begin{proof}
To show that $\Gamma$ is cyclic if and only if it is a homomorphic image of $\mathbb{Z},$ let $\Gamma=\langle a\rangle$ then the map
$$\mathbb{Z}\rightarrow\Gamma,~n\rightarrow a^n$$
is a homomorphism (since $a^{n+m}=a^na^m$ for all $n, m\in\mathbb{Z})$ whose image is $\Gamma.$\\
Conversely, if $f:\mathbb{Z}\rightarrow\Gamma$ is an epimorphism then let $a=f(1).$ Every $\gamma\in\Gamma$ takes the form $\gamma=f(n)$ for some $n\in\mathbb{Z}.$ If $n\geqslant0$ then\\
$$\gamma=f(1+\ldots +1)=f(1)\circ_{\Gamma}\dots\circ_{\Gamma }f(1)=\left(f(1)\right)^n=a^n.$$\\
The same formula holds if $n<0.$ Thus $\Gamma=\langle a\rangle.$
\end{proof}

\begin{priteo}\label{teorema cyclic Hopf}
(cyclic Hopf theorem). In systems with finite cyclic symmetry, generically, Hopf bifurcation at a point $x_0$ occurs with simple eigenvalues, and there exists a unique branch of small-amplitude periodic solutions emanating from $x_0.$ Moreover the spatio-temporal symmetries of the bifurcating periodic solutions are
\begin{equation}
H=\Sigma_{x_0},~H~is~cyclic
\end{equation}
and
\begin{equation}
K=\mathrm{ker}_V(H),~K~is~cyclic,
\end{equation}
and $H$ acts $H-$simply on $V.$ In addition let $\mathbb{Z}_k$ act on $\mathbb{R}^k$ by a cyclic permutation of coordinates. Let $\mathbb{Z}_q\subseteq\mathbb{Z}_n\subseteq\mathbb{Z}_k$. Then there is a $\mathbb{Z}_n-$simple representation with kernel $\mathbb{Z}_q$ with the single exception when $n=k$ is even and $q=\frac{k}{2}.$
\end{priteo}

\begin{proof}
The proof relies on the proof of the homologous Theorem in \cite{Abelian Hopf}, with changes concerning the form of the subgroups $H$ and $K.$ However, we will prefer to give the proof entirely, including the parts that coincide with the proof in \cite{Abelian Hopf}, to easy the lecture of the paper.
We begin as in \cite{Abelian Hopf}, by showing that the equivariant Hopf bifurcation leads to a unique branch of small-amplitude periodic solutions emanating from $x_0.$ From Lemma \ref{lema center subspace} it follows that the bifurcation point $x_0=0$ and therefore $\Gamma=\Sigma_{x_0}.$ Moreover, from Lemma \ref{lema diffeomorphism} it follows that if reducing to the center manifold, we may assume that $\mathbb{R}^n=V$ and therefore from \cite{GS88} it follows that the center subspace $V$ at the Hopf bifurcation point is $\Gamma-$simple. This means that $V$ is either a direct sum of two absolutely irreducible representations or it is itself irreducible but not absolutely irreducible.
Since the irreducible representations of abelian groups (and subsequently cyclic groups) are one-dimensional and absolutely irreducible or two-dimensional and non-absolutely irreducible, it follows that $V$ is two-dimensional and therefore the eigenvalues obtained at the linearization about the bifurcation point $x_0$ are simple. Now the standard Hopf bifurcation theorem applies to obtain a unique branch of periodic solutions.

Let $x(t,\lambda)$ be the unique branch of small-amplitude periodic solutions that emanate at the Hopf bifurcation point $x_0.$ For each $t,$
$$x_0=\lim_{\lambda\rightarrow0}x(t,\lambda).$$
Let $H$ be the spatio-temporal symmetry subgroup of $x(\cdot,\lambda),$ and let $\Phi:H\rightarrow\mathbb{S}^1$ be the homomorphism that associates a symmetry $h\in H$ with a phase shift $\Phi(h)\in\mathbb{S}^1.$ To prove that $H\subset\Sigma_{x_0}$ we have
\begin{align}\label{proof theorem Dp Hopf eq1}
\nonumber
hx_0&=\lim_{\lambda\rightarrow0}hx(0,\lambda)\hspace{0.9cm}\mathrm{by~continuity~of}~h\\
\nonumber
\\
\nonumber
&=\lim_{\lambda\rightarrow0}x(\Phi(h),\lambda)\hspace{0.5cm}\mathrm{by~definition~of~spatio-temporal~symmetries}\\
\nonumber
\\
\nonumber
&=x_0
\nonumber
\end{align}
and therefore $h\in\Sigma_{x_0}.$
In the following we proof that $\Sigma_{x_0}\subset H.$ Let $\gamma\in\Sigma_{x_0}\subseteq\Gamma;$ therefore $\gamma x(t,\lambda)$ is also a periodic solution. Since the periodic is unique (as shown above), we have
$$\gamma\{x(t,\lambda)\}=\{x(t,\lambda)\},$$
so $\gamma\in H.$
Lemma \ref{lema diffeomorphism} allows us to assume that the center manifold at $x_0$ is $\hat{V}=v+x_0,$ which may be identified with $V,$ and therefore $V$ is $H-$invariant. Therefore $V$ is $H-$simple since $\gamma$ is cyclic (and subsequently abelian).
Since $\Gamma$ is cyclic, all its subgroups are cyclic, in particular $H$ and $K.$\\
The proof of the last condition is the proof of Proposition 6.2 in \cite{Abelian Hopf}.
\end{proof}

\section{Constructing systems with cyclic symmetry near Hopf points}\label{section constructing systems}
This section consists in recalling the results corresponding section $4$ in \cite{Abelian Hopf} where the construction of systems with abelian symmetry near Hopf points has been carried out.
When $\Gamma$ is finite cyclic, a key step in constructing H mod K periodic solutions from Hopf bifurcation at $x_0$ is the construction
of a locally $\Sigma_{x_0}-$equivariant vector field.
We first construct, for finite symmetry groups, a $\Gamma-$equivariant vector field that has a
stable equilibrium, $x_0\in\mathbb{R}^n$, with the desired isotropy.
We will use
\begin{lema}\label{lema tisme}
For any finite set of distinct points $y_1,\ldots,y_l,$ vectors $v_1,\ldots,v_l$ in $\mathbb{R}^n$ and matrices $A_1,\ldots,A_l\in \mathrm{GL}(n),$ there exists a polynomial map $g:\mathbb{R}^n\rightarrow\mathbb{R}^n$ such that $g(y_j)=v_j$ and $(dg)_{y_j}=A_j.$
\end{lema}
\begin{proof}
See \cite{Tisme}.
\end{proof}
\begin{priteo}
Let $\Gamma$ be a finite cyclic group acting on $\mathbb{R}^n$ and $x_0\in\mathbb{R}^n.$ Then there exists a $\Gamma-$equivariant system of ODEs on $\mathbb{R}^n$ with a stable equilibrium $x_0.$
\end{priteo}
\begin{proof}
See \cite{Abelian Hopf}.
\end{proof}

In conclusion any point $x_0\in\mathbb{R}^n$ can be a stable equilibrium for a $\Gamma-$equivariant vector field $f:\mathbb{R}^n\rightarrow\mathbb{R}^n.$ It is clear that $(df)_{x_0}$ must commute with the isotropy subgroup $\Sigma_{x_0}$ of $x_0$ \cite{GS03}. The following result states that the linearization about the equilibrium $x_0$ can be any linear map that commutes with the isotropy subgroup.
\begin{priteo}\label{teorema polynomial equivariant}
Let $x_0\in\mathbb{R}^n$ and $A:\mathbb{R}^n\rightarrow\mathbb{R}^n$ be a linear map that commutes with the isotropy subgroup $\Sigma_{x_0}$ of $x_0.$ Then there exists a polynomial $\Gamma-$equivariant vector field $f:\mathbb{R}^n\rightarrow\mathbb{R}^n$ such that $f(x_0)=0$ and $(df)_{x_0}=A.$
\end{priteo}
\begin{proof}
See \cite{Abelian Hopf}.
\end{proof}
When constructing a Hopf bifurcation at points $x_0\in\mathbb{R}^n$ we do not necessarily assume full isotropy. Genericity of $\Sigma_{x_0}-$simple subspaces at points of Hopf bifurcation is given by $\Gamma-$equivariant mappings as follows.
\begin{lema}\label{lema genericity}
Let $\Gamma$ act on $\mathbb{R}^n$ and fix $x_0\in\mathbb{R}^n.$ Let $V$ be a $\Sigma_{x_0}-$invariant neigborhood of $x_0$ such that $\gamma\bar{V}\cap\bar{V}=\emptyset$ for any $\gamma\in\Gamma\backslash\Sigma_{x_0}.$ Let $g:\bar{V}\times\mathbb{R}\rightarrow\mathbb{R}^n$ be a smooth $\Sigma_{x_0}-$equivariant vector field. Then there exists an extension of $g$ to a smooth $\Gamma-$equivariant vector field $f:\mathbb{R}^n\times\mathbb{R}\rightarrow\mathbb{R}^n.$
\end{lema}
\begin{proof}
See \cite{Abelian Hopf}.
\end{proof}

\section{The cyclic Hopf $H~\mathrm{mod}~K$ theorem}
\begin{priteo}
(cyclic Hopf $H~\mathrm{mod}~K$ theorem). Let $\Gamma$ be a finite cyclic group acting on $\mathbb{R}^n.$ There is an $H~\mathrm{mod}~K$ periodic solution that arises by a generic Hopf bifurcation if and only if the following seven conditions hold: Theorem \ref{teorema H mod K} $(a)-(d),$ $H$ is a cyclic isotropy subgroup, there exists an $H-$simple subspace $V$ such that $K=\mathrm{ker}_V(H),$ K is cyclic and let $\mathbb{Z}_k$ act on $\mathbb{R}^k$ by a cyclic permutation of coordinates. Let $\mathbb{Z}_q\subseteq\mathbb{Z}_n\subseteq\mathbb{Z}_k$. Then there is a $\mathbb{Z}_n-$simple representation with kernel $\mathbb{Z}_q$ with the single exception when $n=k$ is even and $q=\frac{k}{2}.$
\end{priteo}
\begin{proof}
Necessity follows from the $H~\mathrm{mod}~K$ theorem (Theorem \ref{teorema H mod K}) and the cyclic Hopf theorem (Theorem \ref{teorema cyclic Hopf}). We'll prove the sufficiency next. The idea of the proof will again, rely heavily on the proof of Abelian Hopf $H~\mathrm{mod}~K$ theorem in \cite{Abelian Hopf}. Let $x_0\in\mathbb{R}^n$ and let $H$ be the isotropy subgroup of the point $x_0,$ ie. $H=\Sigma_{x_0}.$ Moreover, let $W$ be a $H-$simple representation. Since $\Gamma$ is cyclic (in particular, abelian), $W$ is two-dimensional.
Now we can define the linear maps $A(\lambda):W\rightarrow W$ by
$$A(\lambda)=\begin{bmatrix}\lambda&-1\\1&\lambda\end{bmatrix}.$$
Since $W$ is two-dimensional it is easy to prove the commutativity with $A.$
We have
$$A(\lambda)\cdot W=\begin{bmatrix}\lambda&-1\\1&\lambda\end{bmatrix}\cdot\begin{bmatrix}a&-b\\b&a\end{bmatrix}=\begin{bmatrix}\lambda a-b&-\lambda b-a\\a+\lambda b&a\lambda-b \end{bmatrix}=W\cdot A(\lambda).$$
Next we can extend Theorem \ref{teorema polynomial equivariant} to a bifurcation problem as in Lemma \ref{lema genericity}. Let $f:\mathbb{R}^n\times\mathbb{R}\rightarrow\mathbb{R}^n$ be a $\Gamma-$equivariant polynomial such that for all $\gamma\in\Gamma,~f(\gamma x_0,\lambda)=0$ and $(df)_{\gamma x_0,\lambda}|_{\gamma W}=\gamma A(\lambda)\gamma^{-1}.$ Moreover, let $g=f|_{W+x_0}.$ From the way $f$ has been constructed, $g$ is $H-$equivariant on $W+x_0$ and $g(x_0)=0,$ hence from Lemma \ref{lema diffeomorphism} we have that $W$ is the center subspace of $(dg)_{x_0,0.}$

Next consider $(dg)_{x_0,\lambda}|_W;$ its eigenvalues are $\sigma(\lambda)\pm i\rho(\lambda)$ with $\sigma(0)=0,~\rho(0)=1$ and $\sigma'(0)\neq0.$ Then the equivariant Hopf theorem extended to a point $x_0\in\mathbb{R}^n$ implies the existence of small-amplitude periodic solutions emanating from $x_0$ with spatio-temporal symmetries $H$ and spatial symmetries $K.$

\end{proof}
\section{General considerations between the differences of the results in this article and \cite{Abelian Hopf}}
In the first place it must be highlighted that one can start we the methodology used in \cite{Abelian Hopf} and add the restrictions presented in this paper to obtain the Cyclic Hopf $H~ \mathrm{mod}~ K$ Theorem, but not vice-versa. This is obvious, because any cyclic group is abelian but not any abelian group is cyclic.\\

In this section we use the Cyclic Hopf $H~\mathrm{mod}~K$ Theorem to exhibit symmetry pairs $(H,K)$ that are admissible by the Abelian Hopf $H~\mathrm{mod}~K$ Theorem but not admissible by the Cyclic Hopf $H~\mathrm{mod}~K$ Theorem. Let $\mathbb{Z}_l$ act on $\mathbb{R}^l$ by cyclic permutation of coordinates and $\Gamma=\mathbb{Z}_l\times\mathbb{Z}_k$ act on $\mathbb{R}^l\times\mathbb{R}^k$ by the diagonal action, where $l,k>1.$ We show Abelian Hopf $H~\mathrm{mod}~K$ admissible but not Cyclic Hopf $H~\mathrm{mod}~K$ admissible pairs for this action of $\Gamma$ by classifying in Theorem \eqref{equi} all Cyclic Hopf $H~\mathrm{mod}~K$ admissible pairs $K\subset H\subset \Gamma$ and showing that there are admissible pairs that are not on the list.
\begin{priteo}\label{equi}
By applying the Cyclic Hopf $H~\mathrm{mod}~K$ Theorem, the $(H,K)$ Hopf-admissible pairs in $\Gamma$ are $(\mathbb{Z}_m\times\mathbb{Z}_n, ~\mathbb{Z}_m\times\mathbb{Z}_q)$ where $q$ divides $n$ except when $q=\frac{k}{2}$ and $n=k,$ and $(\mathbb{Z}_m\times\mathbb{Z}_n, ~\mathbb{Z}_p\times\mathbb{Z}_n)$ where $p$ divides $m$ except when $p=\frac{l}{2}$ and $m=l.$ Moreover, $m$ and $n$ are coprimes, $m$ and $q$ are coprimes with $m\neq q,$ and $p$ and $n$ are coprimes, with $p\neq n.$
\end{priteo}
\begin{proof}
The proof is a restriction to the cases $m$ and $n$ are coprimes with $m\neq n,$ $m$ and $q$ are coprimes with $m\neq q,$ and $p$ and $n$ are coprimes, with $p\neq n,$ of the proof of Theorem 6.1 in \cite{Abelian Hopf}.
\end{proof}
To find an example of a pair $(H,K)$ that is admissible by the Abelian Hopf $H~\mathrm{mod}~K$ Theorem but not by the Cyclic Hopf $H~\mathrm{mod}~K$ Theorem, let's take $(H=\mathbb{Z}_m\times\mathbb{Z}_n,~K=\mathbb{Z}_m\times\mathbb{Z}_q)$ where $q$ divides $n$ except when $q=\frac{k}{2}$ and $n=k,$ and $m$ and $n$ are not coprimes, $m$ and $q$ are not coprimes. They are admissible by the Abelian Hopf $H~\mathrm{mod}~K$ by applying Theorem 6.1 in \cite{Abelian Hopf}. However, they are not admissible by the Cyclic Hopf $H~\mathrm{mod}~K$ Theorem because of the application of the the Fundamental Theorem of finitely generated abelian groups. Indeed, if, for example $m$ and $n$ are not coprimes then they have a common divisor integer $a\in\mathbb{R}_+$ that is prime, and in this case $m=ab,~n=ac$ for some integers $b\in\mathbb{R}_+,~c\in\mathbb{R}_+$ and the group $\mathbb{Z}_{ab}\times\mathbb{Z}_{ac}$ is not cyclic. A similar case applies for the group $K=\mathbb{Z}_m\times\mathbb{Z}_q$ if $m$ and $q$ are not coprimes with $m\neq q,$ or the group $K=\mathbb{Z}_p\times\mathbb{Z}_n$ if $n$ and $p$ are not coprimes with $n\neq p.$\\
\\

\paragraph{\bf Acknowledgements}
The author would like to thank the helpful suggestions received from the referee, which improved the presentation of this paper. He also acknowledges economical support form FCT grant SFRH/ BD/ 64374/ 2009.


\begin{thebibliography}{00}
\bibitem{GS85}{\sc M. Golubitsky, D.G. Schaeffer}, {\it Singularities and groups in bifurcation theory I}, Applied mathematical sciences \textbf{51}, Springer-Verlag, (1985).
\bibitem{GS88}{\sc M. Golubitsky, I. Stewart, D. G. Schaeffer}, {\it Singularities and groups in bifurcation theory II}, Applied mathematical sciences \textbf{69}, Springer-Verlag, (1988), 388--399.
\bibitem{GS03}{\sc M. Golubitsky, J. Stewart}, {\it The symmetry perspective}, Birkhauser Verlag, (2003), 63--68.
 487--509.
\bibitem{Abelian Hopf}{\sc N. Filipski, M. Golubitsky}, {\it The Abelian Hopf $H~\mathrm{mod}~K$ Theorem}, SIAM J. Appl. Dynam. Sys. \textbf{9}, (2010), 283--291.
\bibitem{Tisme}{\sc P. Lancaster, M. Tismenetsky}, {\it The theory of matrices, 2-nd edition}, Academic Press New-York, (1985).
\end{thebibliography}
\end{document}